\documentclass[11pt,reqno]{amsart}
\usepackage{xifthen,setspace}
\usepackage{hyperref}
\usepackage{pkgfile}

\newcommand{\e}{\varepsilon}

\allowdisplaybreaks

\newtheorem{claim}{Claim}
\newtheorem{theorem}{Theorem}
\newtheorem{proposition}{Proposition}
\newtheorem{lemma}{Lemma}
\newtheorem{corollary}{Corollary}

\newtheorem{definition}{Definition}
\newtheorem{conjecture}{Conjecture}
\newcommand{\er}{{Erd\H{o}s-R\'enyi }}

\newcommand\old[1]{}

\title{Recovery and Rigidity in a Regular Stochastic Block Model}

\author[Brito]{Gerandy Brito}
\address{Department of Mathematics, University of Washington}
\email{gerandy@math.washington.edu}

\author[Dumitriu]{Ioana Dumitriu}
\address{Department of Mathematics, University of Washington}
\email{dumitriu@math.washington.edu}

\author[Ganguly]{Shirshendu Ganguly}
\address{Department of Mathematics, University of Washington}
\email{sganguly@math.washington.edu}

\author[Hoffman]{Christopher Hoffman}
\address{Department of Mathematics, University of Washington}
\email{hoffman@math.washington.edu}

\author[Tran]{Linh V. Tran}
\address{International University, National University Hochiminh City}
\email{tvlinh@hcmiu.edu.vn}

\begin{document}

\maketitle
\begin{abstract}
The stochastic block model is a natural model for studying community detection in random networks.  Its clustering properties have been extensively studied in the statistics, physics and computer science literature.  Recently this area has experienced  major mathematical breakthroughs,  particularly for the binary  (two-community) version, see \cite{MNS12, MNS13, M13}.  
In this paper, we introduce a variant of the binary model which we call the regular stochastic block model (RSBM). We prove rigidity by showing that with high probability an exact recovery of the community structure is possible. Spectral methods exhibit a regime where this can be done efficiently. Moreover we also prove that, in this setting, any suitably good partial recovery can be bootstrapped to obtain a full recovery of the communities.
\end{abstract}

\section{Definition of the model and main results}
\label{definition}
The stochastic block model (SBM) is a classical cluster-exhibiting
random graph model that has been extensively studied, both empirically
and rigorously, across numerous fields. In its simplest form, the SBM
is a model of random graphs on $2n$ nodes with two equal-sized
clusters $\mathcal{A}$ and $\mathcal{B}$ such that
$|\mathcal{A}|=|\mathcal{B}|=n$ and $\mathcal{A} \cap \mathcal{B}
=\emptyset$. Edges between various pairs of vertices appear
independently with probability $p=p_n$ if the two vertices belong to the
same cluster and with probability $q=q_n$ otherwise.
Thus, for any vertex, the expected number of same-class neighbors is
$a:=a_n:=p (n-1) \sim pn$, and the expected number of across-class neighbors is
$b:=b_n:= q n$.

Given a realization of the graph, the broad goal is to determine
whether it is possible (with high probability) to find the partition
$\mathcal{A},\mathcal{B}$; and if the answer is yes, whether it is
possible to do so using an efficient algorithm. Otherwise, the best
one can hope for is the existence of an algorithm that will output a
partition which is highly  (or at least positively) correlated with
the underlying cluster. To this end, consider the space $\mathcal{M}$
of all algorithms which take as input a finite graph on $2n$ vertices and output a partition of the vertex set into two sets. 
Informally, we say that an algorithm in $\mathcal{M}$ allows for {\bf
  weak recovery} if, with probability going to $1$ as $n$ goes to infinity, it outputs a partition $(A',B')$ such that
$|\mathcal{A} \Delta A'|+|\mathcal{B} \Delta B'| =o(n)$  (here $\Delta$ denotes the symmetric difference).
We say that an algorithm allows for {\bf strong recovery} if, with
probability going to $1$ as $n$ goes to infinity, it outputs the
partition $(\A,\B)$. Finally, an algorithm in $\mathcal{M}$ will be called {\bf efficient} if its run time is polynomial in $n$.

The problem of community detection described above is closely related
to the min-bisection problem, where one looks for a partition of the
vertex set of a given graph into two subsets of equal size such that
the number of edges across the subsets is minimal. In general, this problem is
known to be NP-hard \cite{GJS76}; however, if the min-bisection is
smaller than most of the other bisections, the problem is known to be
simpler. This fact was noticed a few decades ago, with the advent of
the study of min-bisection in the context of the SBM. In particular, Dyer and Frieze \cite{dyer1989solution} produced one of the earliest results when they showed that if $p > q$ are fixed as $n \to \infty$ then the min-bisection is the one that separates the two classes, and it can be found in expected $O(n^3)$ time.
Their results were improved by Jerrum and Sorkin
\cite{jerrum1998metropolis} and Condon and Karp
\cite{condon2001algorithms}. Each of these papers were able to find
faster algorithms that worked for sparser  graphs. The latter work was
able to solve the min-bisection problem when the average degrees were
of order $n^{1/2+\epsilon}$. 

Until a few years ago most of the literature on both the min-bisection
problem and community detection in the SBM had focused on the case of
increasing expected degrees (i.e.\ $a, b \to \infty$ as $n \to
\infty$), with the best results at that time showing that if the
smallest average degree is roughly $\log n$, then weak recovery is
possible (e.g., McSherry \cite{mcsherry2001spectral} showed that spectral clustering arguments can work to detect the clusters in this setting).
Recently, the sparse case, i.e. when $a, b=O(1)$ has been the focus of
a lot of interest.  This regime is interesting both from a theoretical
and an applied point of view since a lot of real world networks turn
out to be sparse; for more on this see \cite{DLLM08}. 
 Coja-Oghlan demonstrated a spectral algorithm that finds a bisection which is positively correlated with the true cluster when the average degree is a large constant \cite{coja2010graph}.
Using ideas from statistical physics, Decelle, Krzakala, Moore and Zdeborov\'a gave a precise prediction for the problem of recovering a partition positively correlated with the true partition in the sparse SBM \cite{decelle2011}. The prediction was rigorously confirmed in a series of papers by Mossel, Neeman and Sly \cite{MNS12} \cite{MNS13}, and Massouli{\'e}
  \cite{M13}, where it was shown that this level of recovery  is
  possible iff $(a-b)^2>(a+b)$. More recently, \cite{MNS14} found
  necessary and sufficient conditions for $a$ and $b$ under which
  strong recovery is possible. Before them, Abbe, Bandeira and
  Hall  \cite{abbe2014exact} also characterized strong recovery assuming the edge probabilities to be constant factors of $\frac{\ln(n)}{n}$.

In \cite{MNS12} Mossel, Neeman and Sly proposed two regular versions
of the SBM in a sparse regime, and they conjectured thresholds for the
recovery of a correlated partition for each of the models. They also suggested that spectral methods should help to differentiate between the regular SBM and a random regular graph. 
In this article we study a slightly different version of a regular SBM where in addition to the graph being regular,  the number of neighbors that a vertex has within its own community is also a constant. Formally, we have the following definition.
\begin{definition}\label{model}
For integers $n,d_1$ and $d_2$ denote  by $\mathcal{G}(n,d_1,d_2),$
the random regular graph with vertex set $[2n]$, obtained as follows:
Choose an equipartition (parts have equal sizes) $(\mathcal{A}, \mathcal{B})$ of the vertex set, uniformly from
among the set of such equipartitions. Choose two independent copies of uniform simple $d_1$-regular graphs  with vertex set $\mathcal{A}$, respectively $\mathcal{B}$. Finally, connect the vertices from $\mathcal{A}$ with those from $\mathcal{B}$ by a random $d_2$-bipartite-regular graph chosen uniformly. We refer to this family of measures on graphs as the regular stochastic block model (RSBM).
\end{definition}

The goal of this article is to investigate the similarities and
differences between the RSBM and the classical SBM. For the rest of
the article we assume that  $\min\{d_1,d_2\}\geq 3$.  This assumption
implies that, with high probability, the resulting graph is
connected. This differs from the SBM with bounded average degree,
which has a positive density of isolated vertices, which make strong recovery impossible. 
The  constant degree of all the vertices in the RSBM makes the local
neighborhoods easier to analyze; however, as this model lacks the edge-independence present in the SBM,  some computations become significantly more difficult. 

Throughout the rest of the article we say a sequence of events happen asymptotically almost surely (a.a.s.) if the probabilities of the events go to $1$ along the sequence. The underlying measure will be always clear from context.

Our first result, the next proposition, pertains to the rigidity of
RSBM; it says that the RSBM is asymptotically distinguishable from a uniformly chosen random regular
graph with the same average degree. Below, $||\cdot,\cdot||_{TV}$ denotes the total variation distance between measures.
\begin{proposition}\label{orthogonal}
Let $\mu_n$ be the measure induced by $\mathcal{G}(n,d_1,d_2)$ on the
set $\text{Reg}(2n,d_1+d_2)$ of all $(d_1+d_2)$-regular graphs on $2n$
vertices and let $\mu'_n$ be the uniform measure on the same set
$\text{Reg}(2n,d_1+d_2)$. Then for any positive integers $d_1, d_2
\geq 3$,
$$\lim_{n\to \infty} ||\mu'_{n},\mu_n||_{TV} = 1.$$ 
\end{proposition}
This result sharply contrasts the RSBM and the SBM (which is
indistinguishable from an \er random graph with the same size and
average degrees satisfying $(a-b)^2\leq (a+b)$ \cite{MNS12}).

In order to determine whether it is possible to recover the partition
in the RSBM, one must first answer a basic question about the random
graph $\mathcal{G}(n,d_1,d_2)$: is the `true partition' $(\A, \B)$ 
identifiable. I.e., is $(\A, \B)$ the only way to partition the graph such that the subgraphs on the parts are $d_1$-regular (which then implies that the subgraph across is $d_2$-bipartite)? The following result shows that the answer is yes if $d_1$ and $d_2$ are sufficiently large.  

\begin{theorem}\label{uniqueness}  
There exists a constant $d'>0$ such that, for $d_1>d_2>d'$, $\mathcal{G}(n,d_1,d_2)$ has a unique partition a.a.s.
\end{theorem}

The particular value of $d'$ that we get is far from optimal; we
conjecture that the conclusion of this theorem should be true for
$d'=2$. The proof of Theorem \ref{uniqueness} is quite technical and is given in section \ref{ortho}.
 
To our knowledge, this is the first uniqueness of partition result for
block models with constant degrees. Such a result is not true,
however, in the classical setting where the edges are independent, since with constant probability one has isolated vertices.

If the original partition is unique in most cases then one can, in principle, find the original partition by
exhaustive search, and hence achieve strong recovery. This is again in
sharp contrast with the SBM, where strong recovery is achievable only
in the case of growing degrees. 

The next natural direction is to look for an efficient algorithm for strong recovery. 
While we do not answer this question in general, we do exhibit one regime where such an algorithm exists. 

\begin{theorem}\label{str12}
Assume $(d_1-d_2)^2> 4(d_1+d_2-1)$ and $d_1$ is even. Then there is an efficient algorithm that allows strong recovery.
\end{theorem} 

 The reason for the above asumption on the parity of $d_1$ is that in this case our graph can be viewed as a ``random lifts", allowing us to exploit their spectral properties, see Section \ref{sop2}. 

Nonetheless, we believe spectral arguments can be used to prove weak recovery, with no restriction on the parity of $d_1$, and record the following conjecture.  
\begin{conjecture}\label{weakpart1}
Assume  $(d_1-d_2)^2>4(d_1+d_2-1)$. Then there is an efficient algorithm that allows weak recovery.
\end{conjecture}

Having obtained weak recovery, one can then achieve strong recovery by recursively applying the majority algorithm where one simultaneously updates the label of each vertex by the majority label among the neighbors. That this can be done is again an example of the rigidity in this model, and highlights one of the main differences between RSBM and the classic SBM. It shows that for the former, existence of an efficient algorithm for weak recovery implies the existence of an algorithm for strong recovery. This contrasts with the separate thresholds in the SBM \cite{MNS14}.

We present the majority algorithm in the section below.

\subsubsection{Majority algorithm.}\label{ma1}
Recall that $\mathcal{A}$ and $\mathcal{B}$ are the true
communities. Let $(A,B)$ be any partition (not necessarily an
equipartition) of the vertex set. For each $i \in [2n]$, let $\sigma_i=+1$ if $i\in A$ and $\sigma_i=-1$ if $i\in B$. 
\begin{itemize}
 \item [] $\mathbf{Initialize}$ $A_0=A, B_0=B$.
 \item [] \textbf{For} $i\in [2n]$  (\text{majority rule})\\
									           $\hat \sigma_i=\text{sign}(\displaystyle{\sum_{v_j\sim v_i}\sigma_j})$
 \item[]	\textbf{Return} $A_1=\{v_i: \hat \sigma_i=+1\}$, $B_1=\{v_i: \hat \sigma_i=-1\}$		
\end{itemize}

Similar applications of the majority algorithm appear in \cite{abbe2014exact} and \cite{MNS14}. There, the authors find criteria for both weak recovery and strong recovery in the SBM. It is not hard to see that weak
recovery and strong recovery are not equivalent  in the sparse SBM,
since the presence of isolated vertices prevents strong recovery. 

We will refer to the majority algorithm as \textbf{Majority}.
The following theorem yields strong recovery from weak recovery. 

 \begin{theorem}\label{majority} Assume $d_1>d_2+4$.  Then there
   exists an $\e=\e(d_1)>0$  such that the following is true a.a.s.:
   given a graph $\mathcal{G}(n,d_1,d_2)$ and any partition $(A,B)$
   of its vertex set such that $|A\cap \mathcal{A}|>(1-\e) n$ and
   $|B\cap \mathcal{B}|>(1-\e) n$, \textbf{Majority} recovers the true
   partition $(\mathcal{A}, \mathcal{B})$ if started with $(A,B)$,
   after $O(\log(n))$ iterations. The constant in the $O(\cdot)$ depends on $\e, d_1.$
\end{theorem}

The way we iterate the \textbf{Majority} algorithm will be clear from the proof of Theorem \ref{majority}, see section \ref{proof_majority}.

\section{Main ideas and organization of the paper}\label{psk}

In this section we sketch the main ideas behind the proofs and also the structure of the paper.

\subsection{Organization}

There are five results we present in this paper. In section \ref{ortho}, we prove
Proposition \ref{orthogonal} and Theorem \ref{uniqueness}. This section also contains a review of some standard definitions in the random graph literature that we make use of throughout the paper. We present an informal sketch of the proof of Theorem \ref{uniqueness} in section \ref{uni}, introduce some useful notions on random lifts and multigraphs in section \ref{sop2}, where we explain how to obtain Theorem \ref{str12}. Section \ref{sec_weak}  is concerned with proving Theorem \ref{str12}, while section \ref{algorithm} contains the proofs of Theorem \ref{majority}.


\subsection{Sketch of the proof of Theorem \ref{uniqueness}.} \label{uni}
Recall from Definition \ref{model}, in the graph $G:=\mathcal{G}(n,d_1,d_2)$ on $[2n],$ $(\A,\B)$ form the true partition.

Let us introduce the following notation: for any $V\subset [2n]$ let $G_V$ denote the subgraph induced by $G$ on $V$. For disjoint subsets $V_1,V_2,$ let $G_{(V_1,V_2)}$ denote the subgraph on $V_1 \cup V_2$ induced by the edges in $G$ with one endpoint in $V_1$ and the other in $V_2.$
For any $v\in [2n]$ and $V \subset [2n]$ let $deg_V(v)$ denote the number of edges incident on $v$ whose other endpoint is in $V$.

Thus Theorem \ref{uniqueness} says that, a.a.s., there does not exist any $V\subset [2n]$ with $V\neq \A,\B$  and $|V|=n$ such that the following two conditions hold simultaneously:

\begin{itemize}
\item Both $G_{V}$ and $G_{[2n]\setminus V}$ are $d_1$-regular graphs.
\item $G_{(V,[2n]\setminus V)}$ is a $d_2$-regular bipartite graph.
\end{itemize}
 
However we show that it is even unlikely that  $G_{V}$ is
$d_1$-regular for any $V \neq \A,\B$ with $|V|=n$. 
To this end we fix such a $V$ and  let $V_1:=V\cap \A$, $V_2:=V \cap \B$, and assume $|V_2|=\alpha n$ with $\alpha \le \frac{1}{2}.$
Note that, given $G$, $V$ and $\A$, the degree sequence
$\{deg_{V_1}(v)\}_{v \in V_1}$ is determined; if $G_V$ were
$d_1$-regular graph then for each $v \in V$, $$deg_{V_1}(v)+deg_{V_2}(v)=d_1,$$ 
and hence the degree sequence $\{deg_{V_2}(v)\}_{v \in V_1}$ is also determined, i.e. the number of edges going from each vertex in $V_1$ to $V_2$ is fixed. 

It can be shown using the configuration model (see Section \ref{pm1}
for the definition) that the joint distribution of 
 $\{deg_{V_2}(v)\}_{v \in V_1}$ behaves like i.i.d.\
 $Bin(d_2,\alpha)$'s. The proof now follows by using the above to
 estimate the probability of a certain degree sequence from this
 distribution, and by a union bound over all possible choices of $V.$
We remark that the formal proof involves some case analysis depending on the size of $|V_2|$ and relies on the expansion properties of regular graphs when $|V_2|$ is small.  

\subsection{Sketch of the proof of Theorem \ref{str12}.} \label{sop2}

To prove Theorem \ref{str12}, we make use of the recent work on the spectra of random lifts of graphs in \cite{FK14,B15} and the references therein. For a wonderful exposition of lifts of graphs see \cite{Linial}. We now introduce the notion of lift of a multigraph.

\subsubsection{Random lifts and multigraphs}\label{rlmg}

By a multigraph we simply mean a graph that allows for multiple edges and loops. Next we define the notion of lift. Informally, an $n$-lift of a
multigraph $X=(V,E)$ is a multigraph $X_n=(V_n,E_n),$ such
that for each vertex in $V$ there are $n$ vertices in $V_n$ and
locally both graphs look the ``same''. Formally, let $V_n:=V \times
\{1,2,\ldots n\}$. To define the edge set
in the lift consider the set $S^{E}_n:=\{\pi_{e}\}_{e\in E}$ where
$\pi_{e}\in S_n$ (the set of permutations of $[n]$). We have:

 $$E_n:=\{((x,i),(y,\pi_{e}(i))): e=(x,y)\in E,\,\,\, 1\le i\le n\},$$
for $\pi\in S_{n}^{E}.$
Thus every edge in $E$ ``lifts'' to a matching in $E_n.$  For every $v \in V,$ let $v\times \{1,2,\ldots n\}$ be called the \emph{fiber} of $v.$

A random lift is the lift constructed from $\pi \in S_n^{E}$ where
$\{\pi_{e}\}_{e\in E}$ are chosen uniformly and independently from $S_n.$
Let $A$ and $A_n$ be the adjacency matrices of the multigraphs $X$ and $X_n$, respectively. One can check that all the eigenvalues of $A$ are also eigenvalues of $A_n$ and the corresponding eigenvectors can be ``lifted'' as well to an eigenvector (which is constant on fibers) of the lifted graph.
Let the remaining eigenvalues of $A_n$ be,
\begin{equation}\label{remain1}
|\mu_1| \ge |\mu_2|\ge \ldots \ge |\mu_{r}|,
\end{equation}
where $r=n|V|-|V|.$
With the above definitions we now state one of the main results in \cite{FK14}.
\begin{theorem}\label{spectralnorm} Let $d\ge 3$ be an integer and let $X$ be a finite, $d$-regular multigraph. If $X_n$ is a random $n$-lift of $X$ then, for any $\e>0,$ $$\lim_{n\to \infty}\P(|\mu_1|\ge 2\sqrt{d-1}+\e)= 0~.$$
\end{theorem}

Recall the definition of strong and weak recovery from Section \ref{definition}. We also  need the following definition. 
\begin{definition}\label{notlabel} Let $e:=e_{2n}$ be the vector of all ones of length $2n$. 
Also let $\sigma=\sigma_{2n}$ be the vector of signs which denotes the partition $\A,\B$ i.e. 
$$\sigma(x)=\left\{\begin{array}{cc}
+1 & x \in \A,\\
-1 & otherwise.
\end{array}\right.
$$
\end{definition}

The proof of Theorem \ref{str12} follows by first realizing the
graph $\cG(n,d_1,d_2)$ as a random lift and then using the above
theorem to show spectral separation of $A_n$; moreover, it can be
shown that, with high probability, $\sigma$ in Definition \ref{notlabel} is an eigenvector associated to the second eigenvalue of the lift. The proof of Theorem \ref{str12} is now reduced to finding a good approximation to the unitary eigenvector corresponding to the second eigenvalue. Note that this allows the strong recovery of the partition $(\A,\B)$.

\section{Proof of Proposition \ref{orthogonal} and Theorem \ref{uniqueness}.} \label{ortho}
Let $K_{n}$ be the support of $\mu_n$, i.e., $K_n$ is the set of all graphs which
are $d_1$-regular on $\mathcal{A}$ and $\mathcal{B}$ and $d_2$-regular
and bipartite across, for some equipartition $(\A, \B)$ of $[2n]$.
Let $|\mathcal{G}(n,d)|$  be the number of $d$-regular graphs on $n$ labelled vertices and let $|\mathcal{BG}(n,d)|$ be the number of $d$-regular bipartite graphs on $2n$ vertices. 
To show that $\mu'_n(K_n)\to 0$ we will use the following enumeration results that can be deduced from \cite{MWa03} and \cite{mckay1991asymptotic}. The idea is to count the number of points in the support of the measures $\mu_n$ and $\mu'_n$.
\noindent
We have from \cite[Corollary 5.3]{mckay1991asymptotic} :
\begin{equation}\label{regulargraph}
|\mathcal{G}(n,d)|=C \frac{(nd)!}{(nd/2)!2^{nd/2}(d!)^n}~,
\end{equation}
asymptotically in $n$, where $C=C(n,d)$ remains bounded as $n$ grows. Similarly, from \cite[Theorem 2]{MWa03}:
\begin{equation}\label{biregulargraph}
|\mathcal{BG}(n,d)|=C_1 \frac{(dn)!}{(d!)^{2n}}~,
\end{equation}
asymptotically in $n$, for $C_1=C_1(n,d)$ a bounded function. We have:
\[
\mu'_n(K_n)=\frac{|K_n|}{|\mathcal{G}(2n,d_1+d_2)|}
\]
To compute $|K_n|$, recall Definition \ref{model}, first choose $\mathcal{A}$ and then use \eqref{regulargraph} and \eqref{biregulargraph}. We get: 

\[
\mu'_n(K_n)=C_2{{2n}\choose{n}}\left(\frac{(nd_1)!}{(nd_1/2)!2^{nd_1/2}(d_1)!^n}\right)^2\frac{(nd_2)!}{(d_2)!^{2n}}
\]
\[
\times\frac{(n(d_1+d_2))!2^{n(d_1+d_2)}(d_1+d_2)!^{2n}}{(2n(d_1+d_2))!} 
\]
for $C_2=C_2(n,d_1,d_2)$ bounded as $n$ grows. Using Stirling's Formula we get:
\[
\mu'_n(K_n)= C_3\left(\frac{4{{d_1+d_2}\choose{d_1}}^2d_1^{d_1}d_2^{d_2}}{2^{d_1+d_2}(d_1+d_2)^{d_1+d_2}}\right)^n
\]
\[
= C_3\left(\frac{2{{d_1+d_2}\choose{d_1}}}{2^{d_1+d_2}}\right)^n\left(\frac{2{{d_1+d_2}\choose{d_1}}d_1^{d_1}d_2^{d_2}}{(d_1+d_2)^{d_1+d_2}}\right)^n
\]
Where $C_3$ equals $C_2$ times a universal constant. Both fractions on the right hand side above are less than $1$. This proves Proposition \ref{orthogonal}.
\qed

\subsection{Uniqueness of the clusters}\label{proofT1}

\subsection{Preliminaries}\label{pm1}
For the sake of completeness, we include in this section some of the basic definitions in the random graph literature. Specifically, we define the configuration model to sample random graphs and also the exploration process.
\subsubsection{Configuration model and exploration process}\label{ep1}
The configuration model, introduced by Bender and Canfield
\cite{bendercanfield} and made famous by Bollobas \cite{bollobas}, is
a well known model to study random regular graphs. Assuming that $dn$
is even, the configuration model outputs a $d$-regular  multigraph
with $n$ vertices. This is done by considering an array $\{\xi_{ij},
1\leq i\leq d$,~ $1\leq j\leq n\}$ and choosing a perfect matching of
it, uniformly among all possible matchings. A graph on $n$ vertices is
obtained by collapsing all $\xi_{ij}$ for $1\leq i\leq d$ into a
single vertex, and putting and edge between two vertices $j$ and $t$
for each pair $(\xi_{ij},\xi_{kl})$ present in the matching. We refer to the family $\xi_{ij}$ as \emph{half edges}. 

It is not hard to see that under the condition that the resulting graph is simple, the distribution of the graph is uniform in the set of all simple $d$-regular graphs. Furthermore, it is well known that, for any fixed $d$, as $n$ grows to infinity, the probability that a graph obtained by the configuration model is simple is bounded away from zero. More precisely, denoting by $G$ the resulting graph, one has (see \cite{bollobas}),
\[
\mathbb{P}(G\mbox{ is simple})=(1-o(1))e^{\frac{1-d^2}{4}}.
\]

Thus, to prove a.a.s. statements for the uniform measure on simple $d$-regular graphs it suffices to prove them for the measure induced on multigraphs by the configuration model.

One extremely useful property of this model is the fact that one can construct the graph by exposing the vertices one at a time, each time matching one by one the $d$ half edges of the correspondent vertex, to a uniformly chosen half edge among the set of unmatched half edges. This process will be used crucially in many of the estimates. We include the precise definition for completeness.  
\begin{definition}\label{conf_process}
Consider the following procedure to generate a random $d$-regular graph on $n$ vertices:
 \begin{itemize}
   \item[$-$] Fix an order of the vertices: $v_1< v_2<...< v_n$ and let $\Xi=\{\xi_{ij}\}, 1\leq i\leq d$ and $1\leq j\leq n$, be the set of half edges, where, for any $1\leq j\leq n$, $\xi_{ij}$ are the $d$ half edges incident to vertex $v_j$. Consider the usual lexicographic order on $\Xi$. 
 	 \item[$-$] Construct a perfect matching of $\Xi$ as follows:
           the first pair is $(\xi_{11},\hat{\xi})$ where
           $\hat{\xi}$ is chosen uniformly from $\Xi \setminus
           \{\xi_{11}\}$. Having constructed $k$ pairs, let
           $\xi_{ij}$ be the smallest half edge not matched yet, chose
           $\tilde{\xi}$ uniformly from the set of remaining unmatched
           half edges different from $\xi_{ij}$, and add the edge
           $(\xi_{ij},\tilde{\xi})$.
	 \item[$-$] Output a multigraph $G$, with vertex set $\{v_j\}$ and an edge set induced by the matching constructed in the previous step.
 \end{itemize}
\end{definition}

This construction outputs a graph with the same law as the one given by the configuration model. Conveniently, with this construction we discover all neighbors of vertex $v_1$ first, then we move to $v_2$ and expose its neighbors (it could be the case that some edges are connecting $v_1$ and $v_2$ and those were exposed before!) and so on.  We will refer to this procedure as the exploration process. 
All the above definitions can be easily adapted to  sample bipartite
regular graphs as well, and in this paper we will use both sets of definitions. 

\subsubsection{Proof of Theorem \ref{uniqueness}}\label{pou}

Recall that $d_1>d_2$ and that ($\mathcal{A}$,$\mathcal{B}$) are the true clusters. The idea, as discussed in Section \ref{psk}, will be to show that, conditioned on the choices of $\mathcal{A}$ and $\mathcal{B}$, if we choose another subset of $n$ vertices, the probability of having a $d_1$-regular graph on these $n$ vertices is small. The estimate on the above probability is crucial since it will then allow us to take a union bound over all possible subsets of size $n$ to conclude that, a.a.s., there is a unique pair of clusters.

First we need some definitions.
\begin{definition}\label{bdrydef}Given a graph $G=(V,E),$
\begin{itemize}
\item [i.]
For a vertex $v$ and a set of vertices $S$ denote by $deg_S(v)$ the number of neighbors of $v$ in $S$. 
\item [ii.]
For any subsets $V_1\subset V_2 \subset V$ define the boundary $\partial_{V_2}V_1$ to be the number of edges in $E$ whose one end point lies in $V_1$ and the other in $V_2\setminus V_1.$ 
When $V_2=V$  we use the simpler notation $\partial V_1.$
\end{itemize}
\end{definition}

Consider non-empty subsets $A\subset \mathcal{A}$, $B\subset \mathcal{B}$ such that $|A\cup B|=n$. Without loss of generality assume $|A|\ge|B|$ and let $\alpha$ be such that 
\begin{equation}\label{frac1}
\alpha n=|B|.
\end{equation}
 We will prove Theorem \ref{uniqueness} by showing that given the $d_1$-regular graph with vertex set $\mathcal{A}$, for any choice of $A$ and $B$ the probability that $A\cup B$ is a $d_1$-regular graph goes to zero as $n$ goes to infinity. We use the simple observation that since $\mathcal{A}$ is $d_1$-regular, to have $A\cup B$ $d_1$-regular, for any vertex $v \in A$, the number of neighbors of $v$ in $B$ must be equal to the number of neighbors of $v$ in $\mathcal{A}\setminus A.$ The technical core of the proof involves showing that the probability of this event is small.
 
 We start by proving a lemma. Recall that, in order to have a $d_1$-regular graph with vertex set $A\cup B$ with $A\subset \mathcal{A}$ and  $B\subset \mathcal{B}$ it is necessary that $deg_B(v)=deg_{\mathcal{A}\setminus A}(v)$ for all $v\in A$. For notational brevity let 
\begin{equation}\label{impnotation1} 
 g_v:=deg_{\A\setminus A}(v)
\end{equation} 
  for all $v\in A$. 
\begin{lemma} \label{constant1} Given  $A\subset \mathcal{A},B \subset
  \mathcal{B}$ and a sequence of non-negative numbers $g=(g_1,g_2,
  \ldots, g_{|A|})$ let $$p(g_1,g_2,\ldots,
  g_{|A|}):=\mathbb{P}(deg_B(v)=g_v~ \mbox{for all}~ v\in A).$$ Then,
  for any such $g$, $$\max_{g'} p(g'_1,g'_2,\ldots, g'_{|A|})=
  p(g^*_1,g^*_2,\ldots, g^*_{|A|}),$$ where $g^*_i\in\{\ell,\ell+1\}$
  for some non negative number $\ell=\ell(g)$. The maximum in the
  above is taken over all sequences $g'=(g'_1,g'_2,\ldots, g'_{|A|})$ such that $\sum_{i=1}^{|A|}g_i'=\sum_{i=1}^{|A|} g_i.$
\end{lemma}

The above lemma says that, given the total number of edges going from
$A$ to $B$, the probability of a possible degree sequence is maximized
when all the degrees are essentially the same. Clearly
$l=\left\lfloor{\frac{\sum_{i=1}^{|A|}g_{i}}{|A|}}\right\rfloor$; the
number of $(l+1)$ degrees occurring in $g^*=(g^*_1,g^*_2,\ldots,
g^*_{|A|})$ is determined by $\sum_{i}g^*_i=\sum_{i}g_i.$

\begin{proof} To compute $p(g_1,g_2,\ldots, g_{|A|})$ we use the exploration process for the $d_2-$regular bipartite graph $(\mathcal{A},\mathcal{B})$  where the vertices of $\mathcal{A}$ are exposed  one by one, as sketched in Subsection \ref{ep1}. We order the vertices so that the vertices of $A$ are exposed first. Let $\mathcal{F}_i$ be the filtration generated by the process up to the $i^{th}$ vertex.
Using the exchangeability of the variables $deg_{B}(v_i)$, given a
sequence $\{g_i\}$, w.l.o.g. we can assume $g_1=\min g_i$ and
$g_2=\max g_i$. 

Assume now $g_2-g_1>1$. We will show that $p(g_1,g_2,\ldots,
g_{|A|})<p(g_1+1,g_2-1,\ldots, g_{|A|})$, which implies the lemma. 
We start with the following simple observation:
\begin{align*}
\mathbb{P}(deg_B(v_i)=g_i,~i\geq3 ~ \big | ~\mathcal{F}_2, & deg_B(v_1)=g_1,~deg_B(v_2)=g_2) =\\ 
& \mathbb{P}(deg_B(v_i)=g_i,~i\geq3 ~ \big | ~\mathcal{F}_2, deg_B(v_1)=g_1+1,deg_B(v_2)=g_2-1).
\end{align*}
This is because under the above two conditionings the number of
remaining unmatched half edges in $A,\mathcal{A},B,\mathcal{B}$ is the same. 
Hence it suffices to show that 
\begin{align}\label{bayes}
\mathbb{P}(deg_B(v_1)=g_1, deg_B(v_2)=g_2)< \mathbb{P}(deg_B(v_1)=g_1+1, deg_B(v_2)=g_2-1).
\end{align}

Next we note that
\[
\mathbb{P}(deg_B(v_1)=g_1, deg_B(v_2)=g_2)~=~\binom{d_2}{g_1}\binom{d_2}{g_2}\frac{(\alpha nd_2)_{[g_1+g_2]}((1-\alpha)nd_2)_{[2d_2-g_1-g_2]}}{(nd_2)_{[2d_2]}},
\]
where $(x)_m$ is the falling factorial
$(x)_{[m]}=x(x-1)\ldots(x-m+1)$. To see the above, we first choose
those half edges of $v_1$ and $v_2$ that will connect to half edges in
$B$. Then we choose the $2d_2$ half edges in $\mathcal{B}$ that will
match with the corresponding half edges of $v_1$ and $v_2$ such that
exactly $g_1+g_2$ are incident on vertices in $B$. 

Substituting now into \eqref{bayes} we have:
\begin{align*}
p(g_1,g_2,...g_{|A|})< p(g_1+1,g_2-1,...g_{|A|}) &\Longleftrightarrow
\binom{d_2}{g_1}\binom{d_2}{g_2}< \binom{d_2}{g_1+1}\binom{d_2}{g_2-1}\\
& \Longleftrightarrow
(g_1+1)(d_2-g_2+1)< g_2(d_2-g_1)\\
& \Longleftrightarrow g_1-g_2+1<d_2(g_2-g_1-1)~,
\end{align*}
which follows immediately from $g_2>g_1+1$.
\end{proof}
\noindent

Recall that we are interested in the probability that $A\cup B$ is
$d_1$-regular for a fixed choice of $A$ and $B$. As already discussed, 
\begin{equation}\label{checkpoint}
\mathbb{P}(A\cup B~ \mbox{is $d_1$-regular})\leq \mathbb{P}(deg_{\mathcal{A}\setminus A}(v)=deg_B(v),\forall v\in A).
\end{equation}

Our next goal is to bound the probability of such an event. To this
end we recall the notion of stochastic dominance. 

Let  $\nu_1$ and $\nu_2$ be two probability measures on $\mathbb{Z}$,
and let $X \sim \nu_1,\,Y \sim \nu_2.$ We use $X \preceq Y$ to denote that $\nu_2$ stochastically dominates $\nu_1$. 

Recall now Definitions \ref{model} and \ref{bdrydef}, as well as
\eqref{impnotation1}.
\begin{lemma}\label{domination} Let $M=\min \{\partial_{\A} A,n/2\}$,
and let $Y=(Y_1,Y_2,\ldots, Y_{M})$ where $Y_i\sim Bin(d_2,2\alpha)$
are i.i.d.. Then
\[
\mathbb{P}(deg_B(v)=g_v~,~\forall\, v\in A~ \big |~\A) ~\le~ \prod_{i=1}^{M}\P(Y_i\ge 1)~.
\]
\end{lemma}  

For notational brevity, we have denoted by $\mathbb{P}(\cdot\mid \mathcal{A})$ the random graph measure $\mathcal{G}(n,d_1,d_2)$ conditioned on the subgraph induced by $\A$.
\begin{proof}
First recall that by Lemma \ref{constant1} the quantity on the left
hand side is maximized when for all $v$,  $g_v\in \{\ell,\ell+1\}.$ Hence we assume that this is the case.
Now to prove the lemma we consider the exploration process defined
above. The definition requires us to fix an order on the vertices of
$\mathcal{A}$; we do this in the  following way. Consider the  two cases: 
\begin{itemize}
\item [$i$.]$\ell=0:$
First come all the vertices $v_i \in A$ with $g_i=1$, followed by the remaining vertices in $A$. Then come all the vertices in $\mathcal{A}\setminus A$.
\item [$ii$.] $\ell>0:$
First come all the vertices $v_i \in A$ with $g_i=\ell$, followed by the remaining vertices in $A$. Then come all the vertices in $\mathcal{A}\setminus A$. 
\end{itemize}
Recall that $\mathcal{F}_i$ is the filtration up to vertex $i.$
Note that, for $1\leq i \leq \min (\partial_{\mathcal{A}} A,n/2)$, 
\[
deg_B(v_i) \big |\mathcal{F}_{i-1}~\preceq~Bin\left(d_2,\frac{\alpha nd_2-(i-1)}{nd_2-id_2}\right).
\]
This follows from the simple observation that for any of the cases
mentioned above for the $i^{th}$ vertex, there are at most $(\alpha
nd_2-(i-1))$ half edges in $B$ that haven't yet been matched. 
Now note that since by hypothesis $i\le \frac{n}{2},$
\begin{align*}
\frac{\alpha nd_2-(i-1)}{nd_2-id_2}& \le \frac{\alpha n d_2}{nd_2/2}\\
&= 2\alpha. 
\end{align*}
Thus we are done.
\end{proof}

As already used in the proof of the above lemma, 
\[
\mathbb{P}(A\cup B~\mbox{is $d_1$-regular}\mid \mathcal{A})\leq p(\ell,\ell,\ldots,\ell,\ell+1,\ldots,\ell+1)\] for some $\ell=\ell(\mathcal{A},A)$. 
\noindent
In case $i.$ we see that by Lemma \ref{domination}
\begin{align}
\label{exp1}
p(0,0,\ldots,0,1,\ldots,1)=p(1,1,\ldots, 1,0,\ldots,0) & \le \prod_{i=1}^{\min \{n/2,\partial_{\mathcal{A}}A\}}\mathbb{P}(Y_i\ge 1) \\
\nonumber 
& \le  \prod_{i=1}^{\min\{n/2,\partial_{\mathcal{A}}A\}} (2 d_2\alpha)
\end{align}
The first equality follows by exchangeability. The first inequality follows from Lemma \ref{domination}. The second is a simple consequence of the fact that for a nonnegative variable the probability of being bigger than $1$ is at most its expectation.

In case $ii$ by similar arguments 
\begin{align}
\label{exp2}
p(\ell,\ell,...\ell,\ell+1,...,\ell+1)&\le \prod_{i=1}^{n/2}\mathbb{P}(Y_i\ge 1)\\
\nonumber
& \le  \prod_{i=1}^{n/2} (2 d_2\alpha).
\end{align}
Note that in \eqref{exp2} the term $\partial_{\mathcal{A}}A$ does not appear. This is because in this case by hypothesis $$|\partial_{\mathcal{A}}A|\ge \ell |A|\ge \frac{n}{2}.$$
To proceed with the proof of Theorem \ref{uniqueness} we quote two standard results on the expansion of random $d$-regular graphs.
Let $\gamma$ be the spectral gap for the operator of the random walk in the uniform random regular graph $G\in \mathcal{G}(n,d)$, i.e.:
\begin{equation}\label{gapdef}
\gamma=1-\frac{\lambda_2}{d}
\end{equation}
where $\lambda_2$ is the second largest eigenvalue of the adjacency matrix of $G$.

\begin{theorem}\cite[Theorem 1.1]{F08} With probability going to $1$ as $n \to \infty,$ $$\gamma \ge 1-\frac{2}{\sqrt{d}}.$$
\end{theorem}
The next result was proven independently in \cite{LS88} and \cite{JS89}. We will use it as it appears in \cite[Theorem 13.14]{LPW}.
\begin{theorem}\label{min_cut} Let $G$ be a $d$-regular graph in $n$ vertices. For any $S \subset V(G),$ with $|S|\le \frac{n}{2},$
$$\frac{\gamma}{2}\le \frac{|\partial S|}{d|S|}.$$
\end{theorem}

Putting everything together we get the following:
For $d_1\ge 16$, a.a.s., for all $S\subset \mathcal{A} $ with $|S|\le \frac{n}{2}$ $$|\partial_{\mathcal{A}} S|\ge \frac{d_1}{4}|S|.$$
In particular since $|A|\ge n/2$ it follows that, a.a.s.,
\begin{equation}\label{expans1}
|\partial_{\A} A|=|\partial_{\A} (\mathcal{A}\setminus A)|~\ge~ \frac{d_1}{4}|\mathcal{A}\setminus A|.
\end{equation}

In case $i.$ ($\ell=0$) plugging \eqref{expans1} in \eqref{exp1} we get

\begin{align}\label{otra_exp}
\mathbb{P}(A\cup B~\mbox{is $d_1$-regular}| \mathcal{A}) &\le \prod_{i=1}^{\min(n/2,|\partial_{\A} A|)}\mathbb{P}(Y_i\ge 1) ~\le~ \prod_{i=1}^{\frac{|\partial_{\A} A|}{2}}\mathbb{P}(Y_i\ge 1)\\
\nonumber 
&\le \prod_{i=1}^{\frac{d_1}{8}\alpha n}\mathbb{P}(Y_i\ge 1)
\end{align}
assuming that the $d_1$-regular graph on $\mathcal{A}$ satisfies \eqref{expans1}. The second inequality follows from the simple observation that since $\ell=0,$  we have $|\partial A| \le n.$

Recall that we want an upper bound on the right hand side of \ref{checkpoint}. Combining Lemma \ref{domination}, \eqref{exp2} and \eqref{otra_exp} we get
\begin{align}\label{upbnd1}
\mathbb{P}(A\cup B~\mbox{is $d_1$-regular}\mid \mathcal{A})~\leq~  \mathbb{P}(Y\ge 1)^{\frac{d_1}{8}\alpha n}+\mathbb{P}(Y\ge 1)^{n/2}.
\end{align}
The two terms on the right hand side correspond to the two cases $\ell=0 $ and $\ell \ge 1$.

Next we show that the bounds in \eqref{upbnd1} are good enough to be able to use union bound over all possible choices of $A$ and $B.$ 
There are  $\binom{n}{\alpha n}^2$ ways to choose $A$ and $B$. Denote
by $R_{\alpha}$ the event that $A\cup B$ is $d_1$-regular for \textit{at least} one choice of $A$ and $B$. Thus by union bound,
\begin{align}\label{finalexp1}
\mathbb{P}(R_{\alpha})& \leq \binom{n}{\alpha n}^2\left[\mathbb{P}(Y\ge 1)^{\frac{d_1}{8}\alpha n}+\mathbb{P}(Y\ge 1)^{n/2}\right].
\end{align}

We now estimate the right hand side using Stirling's formula. 
Let $$H(x)=-x\log x-(1-x)\log(1-x)$$ be the binary entropy function. 
Then the two terms in the right hand side of \eqref{finalexp1}
are at most
$$\frac{2^{n[2H(\alpha)+\frac{d_1}{8}\alpha\log(\mathbb{P}(Y\ge 1))]}}{\sqrt{\alpha n}} \ \ \ \ \text{and} \ \ \ \ \frac{2^{n[2H(\alpha)+\frac{\log(\mathbb{P}(Y\ge 1))}{2}]}}{\sqrt{\alpha n}},$$
up to universal constants involved in Stirling's approximation. Our goal would be to upper bound the two exponents,
\begin{equation}\label{twocase1}
2H(\alpha)+\frac{d_1}{8}\alpha\log(\mathbb{P}(Y\ge 1))~~~\mbox{and}~~~2H(\alpha)+\frac{\log(\mathbb{P}(Y\ge 1))}{2}.
\end{equation}
Recall that $\alpha$ was defined in \eqref{frac1}.
Consider the three following cases:

\textbf{CASE $1$:} 
$\alpha \le \frac{1}{d^2_2}$.

In this case we will use the bound $\mathbb{P}(Y\ge 1)\le 2d_2 \alpha$  by Lemma \ref{domination}.
Plugging this in \eqref{twocase1} we get the following upper bounds 
$$2H(\alpha)+\frac{d_1}{8}\alpha\log(2d_2 \alpha)~~~\mbox{and}~~~2H(\alpha)+\frac{\log(2d_2 \alpha)}{2}.$$
Now,
\begin{align*}
2H(\alpha)+\frac{d_1}{8}\alpha\log(2d_2 \alpha)& =  -2\alpha \log(\alpha)+\frac{d_1}{8}\alpha\log(2d_2 \alpha) -2(1-\alpha)\log(1-\alpha)\\
&\le  \alpha \log(\alpha)(\frac{d_1}{32}-2)-2(1-\alpha)\log(1-\alpha)\\
& \le \alpha \log(\alpha)(\frac{d_1}{32}-4).
\end{align*}
To see the above inequalities first note that since $\alpha \le \frac{1}{d_2^2}$, $\log(2d_2 \alpha) \le \frac{\log(\alpha)}{4} $ as soon as $d_2\ge 4,$ and also $|(1-\alpha)\log(1-\alpha)|\le 4\alpha.$
Similarly for large enough $d_2$ we have 
\begin{align*}
2H(\alpha)+\frac{\log(2d_2 \alpha)}{2}& =  -2\alpha \log(\alpha)+\frac{\log( \alpha)}{8} -2(1-\alpha)\log(1-\alpha)\\
& \le \frac{\log( \alpha)}{16}~.
\end{align*}

Thus for large enough $d_2\le d_1$ $$\mathbb{P}(R_{\alpha})\le \frac{2^{3\alpha \log(\alpha)n}}{\sqrt{\alpha n}}. $$
Hence \begin{eqnarray}
\mathbb{P}\left(\bigcup_{\alpha\in I_1}R_{\alpha}\right)
 &\le& \sum_{\alpha\in I_1}\frac{2^{3\alpha \log(\alpha)n}}{\sqrt{\alpha n}}\nonumber \\
 &\le& n 2^{-3\frac{1}{n}\log(n)n} \nonumber \\ 
 &\le&  \frac{1}{n}, \label{pop}
\end{eqnarray}
where $\alpha\in I_1=(0,\frac{1}{d_2^2})$. The last term is derived using the following:
 The function $\alpha \log \alpha $ is decreasing from $0$ to $1/2$ and the least possible value of $\alpha=\frac{1}{n}$. Plugging this value of $\alpha$  we get the above. 

\textbf{CASE $2$:}  $\frac{1}{d^2_2}\le \alpha \le \frac{C}{d_2}.$ 

Now clearly in this range of $\alpha$, by stochastic domination $\mathbb{P}(Bin(d_2,\alpha)\ge 1)$ is maximized when $\alpha=\frac{C}{d_2}.$
We now use the Poisson approximation of $Bin(d_2,\frac{2C}{d_2})$ to bound the probability $\mathbb{P}(Y\ge 1)$ by a universal constant $c$ which is a function of $C$ for  all $\alpha $ in this range.
Using this, we rewrite \eqref{finalexp1} to get 
\begin{align*}
2H(\alpha)+\frac{d_1}{8}\alpha\log(c)&\le  -2\alpha \log(\alpha)+\frac{d_1}{8}\alpha\log(c) -2(1-\alpha)\log(1-\alpha)\\
&\le -4\alpha \log (\alpha)+\frac{d_1}{8}\alpha \log(c)\\
&\le -5\alpha 
\end{align*}  
for large enough $d_1.$
Similarly for large enough $d_2$ we have 
\begin{align*}
2H(\alpha)+\frac{\log(c)}{2}& \le \frac{\log(c)}{4}.
\end{align*}

Plugging in we get
\begin{align}\label{fexp1}
\mathbb{P}\left(\bigcup_{\alpha\in I_2}R_{\alpha}\right)\le \sum_{\alpha\in I_2}\frac{2^{-5\alpha n}}{\sqrt{\alpha n}}
& \le n 2^{-\frac{5}{d^2_2}n},
\end{align} 
where $I_2=[\frac{1}{d_2^2},\frac{C}{d_2}]$. 
Thus the proof for the case when $\alpha \le \frac{C}{d_2}$ is complete.

\textbf{CASE $3$:} $\frac{C}{d_2} \leq \alpha \leq \frac{1}{2}.$\\

We first need a preliminary lemma. For $d_2 \in \mathbb{N}$ and $\alpha \in (0,1)$ let $Z_{d_2,p}\sim Bin(d_2,p).$
\begin{lemma}\label{supbound}There exists a constant $C_1$ such that for all large enough $d_2$  
$$\sup_{p \in (\frac{C_1}{d_2}, \frac{2}{3}) }\sup_{1\le i\le d_2}\mathbb{P}(Z_{d_2,\alpha}=i)\le \frac{1}{400}.$$
 \end{lemma}
 
\begin{proof} It is a standard fact that for any $d_2,\alpha$
$$\sup_{1\le i\le d_2}\mathbb{P}(Z_{d_2,\alpha}=i) =\mathbb{P}(Z_{d_2,\alpha}=\lfloor{(d_2+1)\alpha \rfloor}).$$ 
Let $k=\lfloor{(d_2+1)\alpha \rfloor}.$
We now estimate $$\mathbb{P}(Z_{d_2,\alpha}=k)={d_2 \choose k} \alpha^{k}(1-\alpha )^{d_2-k}.$$ Since $k>C_1$ by hypothesis using Stirling's formula we have 
\begin{align*}
\mathbb{P}(Z_{d_2,\alpha}=k)&=O\left(\frac{1}{\sqrt{k}}2^{H(\alpha)d_2}2^{-H(\alpha)d_2}\right)\\
&=O\left(\frac{1}{\sqrt{C_1}}\right) \le \frac{1}{400} 
\end{align*}
 for large enough $C_1.$ 
\end{proof}

We now need another lemma. Consider the exploration process for
sampling the bipartite regular graph given by
$\mathcal{A},\mathcal{B}$ (sketched in  Definition
\ref{conf_process}), where vertices of $\mathcal{A}$ are exposed one
by one to find out the neighbors in $\mathcal{B}.$ We do this first
for each half edge incident to the vertices in $A$,  followed by the half edges corresponding to the rest of the vertices in $\mathcal{A}.$ Let us parametrize time by the number of half edges. Consider the Bernoulli variable
\begin{equation}\label{berno1}
B_t=\mathbf{1}(\mbox{the $t^{th}$ half edge is matched to a half edge in }B).
\end{equation}
Now note that the first $d_2$ half edges correspond to $deg_{B}(v_1),$ the second $d_2$ half edges correspond to $deg_{B}(v_2),$ and so on. 
We now make a simple observation that the Bernoulli probabilities do not change much from time $t$ to $t+d_2.$ This then shows that $deg_{B}(v_i)$ are essentially Binomial variables with probability depending on the filtration at time $(id_2).$
Formally, we have the following lemma: let $\mathcal{F}_i$ be the filtration generated up to time $(id_2)$ (when all the half edges up to vertex $i$ have been matched).
\begin{lemma}\label{totalvar1}For any $i\le \frac{n}{4}$ there exists a $p_i$ which is $\mathcal{F}_{i-1}$-measurable such that $$||deg_{B}(v_i)| \mathcal{F}_{i-1},Bin(d_2,p_i)||_{TV}= O\left(\frac{1}{n}\right),$$  where $||\cdot, \cdot||_{TV}$ denotes the total variation norm and the constant in the $O(\cdot)$ notation depends only on $d_2.$
\end{lemma} 
 
\begin{proof}To show this first note that the random variables $B_{t}$ in \eqref{berno1}  are Bernoulli variables with probability $$\hat p _t=\frac{\alpha n d_2 -\sum_{j\le t-1}B_j}{nd_2-t}~.$$  Then clearly   
for all $t\le \frac{nd_2}{4},$ $|\hat p _t-\hat p _{t-1}|\le \frac{4}{n}~.$
The proof thus follows since $$deg_{B}(v_i)=\sum_{(i-1)d_2<j\le id_2}B_j~.$$
\end{proof} 
\noindent
Recall $\ell$ from Lemma \ref{constant1}. Now suppose $A\cup B$ is $d_1-$regular. Then by definition 
\begin{align*}
\ell |A|\le \sum_{i=1}^{|A|}deg_{B}v_i\le d_2|B|& =\alpha n d_2 \\
\implies \ell \le \frac{\alpha}{1-\alpha}d_2 \le 2\alpha d_2.
\end{align*}
\noindent 
Using the above we get that for all $j\le \frac{n}{4}$:
\begin{equation}\label{lb1}
\frac{\alpha nd_2-j(\ell+1)}{nd_2-jd_2}\ge\frac{\alpha nd_2-\frac{n}{4}(3\alpha d_2)}{nd_2}\ge  \frac{\alpha}{4}.
\end{equation}

Above we used the fact that  $\ell+1 \le 2\alpha d_2+1\le 3\alpha d_2$ since $\alpha d_2>C>1$ by hypothesis. 
Also clearly for $j\le n/4$,  since $\alpha\le 1/2,$
\begin{equation}\label{ub1}
\frac{\alpha nd_2-j\ell}{nd_2-jd_2}\le 2/3.
\end{equation}
Assume that all the $deg_{B}(v_i)\in \{\ell,\ell+1\}.$ We have the following corollary.
\begin{corollary}\label{binomapprox}For all $1\le i\le n/4,$ if $deg_{B}(v_j)\in \{\ell,\ell+1\},$ for some $\ell \le 2d_2 \alpha$  for all $j\le i$ then there exists $p_{i}$ which is $\mathcal{F}_{i-1}$ measurable such that $$||deg_{B}(v_i),Bin(d_2,p_{i})||_{TV}= O\left(\frac{1}{n}\right)$$  where $\frac{\alpha}{4} \le p_{i}\le 2/3.$
\end{corollary}
\begin{proof}The proof is immediate from \eqref{lb1}, \eqref{ub1} and Lemma \ref{totalvar1}. 
\end{proof} 
We now complete the proof of Theorem \ref{uniqueness} in the case $\alpha\in I_3=[\frac{C}{d_2},\frac{1}{2}].$
Using the same notation we used before we have:
\begin{eqnarray}
\nonumber
\mathbb{P}\left(\bigcup_{\alpha\in I_3}R_{\alpha}\mid \A\right) & \le & \sum_{\alpha\in I_3}\sum_{A,B}\mathbb{P}(deg_B(v_i)=g_i)\\
\nonumber
& \le &\sum_{\alpha\in I_3} \binom{n}{\alpha n}^2 \frac{1}{400^{n/4}}\\
\nonumber
& = &\sum_{\alpha\in I_3} \frac{1}{{\alpha n}}2^{2H(\alpha)n}\frac{1}{400^{n/4}}\\
& \leq & n\frac{2^{2n}}{400^{n/4}}. \label{fexp2}
\end{eqnarray}
The first inequality is by the union bound. To see the second inequality observe first that by Lemma \ref{constant1} it suffices to assume that $g_i's \in \{\ell,\ell+1\}$. Thus the second inequality follows by Corollary \ref{binomapprox} and Lemma \ref{supbound} as soon as $$\frac{\alpha}{4}\ge \frac{C_1}{d_2}$$ which we ensure by choosing $C\ge 4C_1.$  

Thus combining \eqref{pop}, \eqref{fexp1} and \eqref{fexp2} we have shown that 
$$\mathbb{P}(\cup R_{\alpha}) \le \tau^{n}$$ for some $\tau=\tau(d_2)<1.$  
Hence we are done. 
\qed

\section{Theorem \ref{str12} and connection to the min-bisection problem}\label{sec_weak}

Throughout this section we always assume $d_1$ is even.
We first remark that, under the hypothesis of Theorem \ref{str12}, one can make a quick and simple
connection to the min-bisection problem. It turns out that, in the
case of the RSBM, the two problems are equivalent. More precisesly, in the proof of Theorem \ref{str12} below, we show that the second eigenvalue of $\mathcal{G}(n,d_1,d_2)$ equals $d_1-d_2$ with high probability, which implies that $\gamma=\frac{2d_2}{d_1+d_2}$ where $\gamma$ is the spectral gap defined in \eqref{gapdef}. Hence, it follows by Theorem \ref{min_cut}, that the size of the min bisection of $\mathcal{G}$($n,d_1,d_2$) is at least $n d_2$. Since the true partition $(\A,\B)$ matches this lower bound, it solves the min-bisection problem.

We now proceed towards proving Theorem \ref{str12}. 
 Recall the notion of random lifts from Section \ref{rlmg}.
We will now connect $\cG(n,d_1,d_2)$ (RSBM)  with random lifts of a certain small graph. 
Consider the following multigraph on two vertices: $u$ and $v$, with $d_2$ edges between $u$ and $v$ and $d_1/2$ self loops at both the vertices (recall that $d_1$ is even).
\begin{figure}[hbt]
\includegraphics[scale=.8]{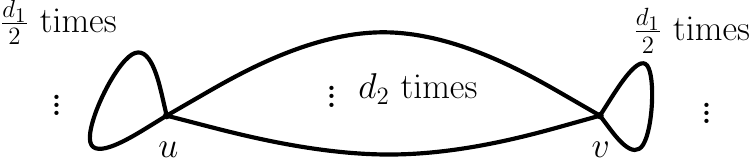}
\caption{Multigraph lifting to $\cG(n,d_1,d_2)$.}
\label{down1}
\end{figure}
To randomly $n-$lift the above graph according to Section \ref{sop2} we choose uniformly $d_1+d_2$  many permutations: 
\begin{equation}\label{genperm1}
\pi_1,\pi_2,\ldots \pi_{d_1},\pi'_1,\pi'_2,\ldots \pi'_{d_2}
\end{equation}
 from $S_n.$

Let the lift be  $\sG(n,d_1,d_2)$ on the vertex set $\{u,v\}\times \{1,2,\ldots n\}$. We naturally identify it with $[2n]=\{1,2,3,4,\ldots 2n\}$ with the first $n$ numbers corresponding to $u\times \{1,2,\ldots n \}$ and the rest corresponding to $v \times \{1,2,\ldots n\}.$

Note that $\sG_1,$ the subgraph induced by $\sG(n,d_1,d_2)$ on $[n]$ has edge set $(i,\pi_{j}(i))$ for $i\in [n]$ and $j\in [d_1/2]$. Similarly $\sG_2,$ on $[2n]\setminus [n]$ has edges $(n+i,n+\pi_{j}(i))$ for $i\in [n]$ and $j\in [d_1]\setminus [d_1/2]$.
The edges between $[n]$ and $[2n]\setminus [n]$ are the edges $(i,n+\pi'_{j}(i))$ for $i\in [n]$ and $j\in [d_2].$
Recall $\cG(n,d_1,d_2)$ from Definition \ref{model}.  A standard model to generate regular graphs is the well known configuration
model, as also used in this article (see Section \ref{pm1})
Now notice that $\sG(n,d_1,d_2)$ is essentially the same as
$\cG(n,d_1,d_2)$ except the graphs are now generated using
permutations in \eqref{genperm1}.  This is known as the Permutation
model (see \cite{F08} and the references therein). 
We now use a well known result which says that the two models are
contiguous, i.e. any event occurring a.a.s. in one of the models occurs
a.a.s. in the other one as well (see \cite{GJKW02}).
\footnote{\cite[Theorem $1.3$]{GJKW02} actually shows contiguity of regular graphs under configuration model and the permutation model. Note that $\cG(n,d_1,d_2)$ and $\sG(n,d_1,d_2)$ are  constructed  from three independent regular graphs constructed using the configuration model and the permutation model. Since contiguity is preserved under taking product of measures, $\cG(n,d_1,d_2)$ and $\sG(n,d_1,d_2)$ are contiguous.}

We now prove Theorem \ref{str12}. Let the graph in Figure \ref{down1} be called $\sC$. The adjacency matrix of $\sC$ is $A_*:=\left[\begin{array}{cc}d_1 & d_2\\
d_2 & d_1 
\end{array}
\right]$ with eigenvalues $d_1+d_2$ and $d_1-d_2$ and corresponding eigenvectors $\left(\begin{array}{c}1\\
1 
\end{array}
\right)$ and $\left(\begin{array}{c}1\\
-1 
\end{array}
\right)$,
respectively. Let $A_{*,n}$ be the adjacency matrix of $\sG(n,d_1,d_2)$ which as discussed above is a random $n-$lift of $\sC.$ From the discussion in Section \ref{sop2} we have the following: 
\begin{itemize}
\item $d_1+d_2$ and $d_1-d_2$  are eigenvalues of $A_{*,n}$, with eigenvectors $e$ and $\sigma$ respectively (see Definition \ref{notlabel}). 
\item By Theorem \ref{spectralnorm}, for any $\e>0$, a.a.s., all the other eigenvalues $\lambda$ of $A_{*,n}$ satisfy $|\lambda|\le 2\sqrt{d_1+d_2-1}+ \e.$
\end{itemize}
Let $A_{n}$ be the adjacency matrix of $\cG(n,d_1,d_2).$ That the
first fact above holds for $A_n$ as well is easy to check. Moreover,
using the contiguity of the two models, $A_{n}$ also has the second
property a.a.s..
Note that  finding the partition $(\mathcal{A},\mathcal{B}),$ in  Definition \ref{model} is equivalent to finding $\sigma,$ (the eigenvector corresponding to the eigenvalue $d_1-d_2$). Now under the hypothesis of Theorem \ref{str12}, by the above discussion we see that $d_1-d_2$ is the second eigenvalue which is also separated from the first and rest of the eigenvalues. 
Thus, we can efficiently compute a unitary eigenvector, $w$, associated to this eigenvalue. To assign the communities, put $v\in \mathcal{A}$ if and only if $w_v>0$. Strong recovery is then achieved. This proves Theorem \ref{str12}.

\section{Complete reconstruction from partial reconstruction: proof of Theorem \ref{majority}} \label{algorithm}

In this section we prove Theorem \ref{majority}. The idea is to show
that, because of the rigid nature of the graph, if we initialize the
partition with a large number of vertices labeled correctly, one can
bootstrap to deduce the true labels of even more vertices in the next
step. We do this by looking at the majority of a vertex'
neighbors. Recall \textbf{Majority}  from Section \ref{ma1}. We prove
that with high probability the graph $\mathcal{G}(n,d_1,d_2)$ is such
that if the input $(A,B)$ has a large overlap with the true partition
$(\A,\B)$, then one round of the algorithm reduces the number of
wrongly labeled vertices by a  constant factor. Thus it follows then
that, with high probability, after $O(\log(n))$ iterations, no further corrections can be made
and the algorithm outputs the true communities.

\begin{lemma}\label{maj_lem}  
Assume $d_1>d_2+4$ and let
$1/2<\lambda<1$. Then there exists an $\epsilon=\epsilon(d_1)>0$ such that, with probability
$1-O(n^{1/2-\lambda})$, the graph has the property that if  $(A,B)$ (the input)  satisfies $\min\{|A\cap \mathcal{A}|,|B\cap
\mathcal{B}|\}>(1-\epsilon)n$ and if $|\mathcal{A}\bigcap B|=:k$ and $|\mathcal{B} \bigcap A|=:k'$, then 
\[
|\mathcal{A}\bigcap B_1|\leq \lambda  k~\mbox{and}~
|\mathcal{B}\bigcap A_1|\leq \lambda k'~.
\] 
where $(A_1,B_1)$ is the output after one round of \textbf{Majority}.
\end{lemma}
The constant in $O(\cdot)$ depends on $d_1,\lambda, \epsilon.$

\begin{proof}
Let $v \in \mathcal{A}\bigcap B_1$ (that is, $v$ has the
wrong label after one iteration of \textbf{Majority}). We claim that $v$ has more than two neighbors in $\mathcal{A}\bigcap B$, otherwise $v$ will have at least $d_1-2$ neighbors in $\mathcal{A}\bigcap A$ and hence its label will be the sign of:
\[
\sum_{i\sim v} \sigma^1_i\geq d_1-2-(d_2+2)> 0~,
\] 

which contradicts the assumption that $v\in
\mathcal{A}\bigcap {B}_1$. 
Thus the occurrence of the event $\left | \mathcal{A}\bigcap {B}_1
\right | \geq \lambda k$ implies the occurrence of the event
\begin{center}
$E_k:=\{\exists \text{ a subset } S \subset \mathcal{A}, ~ |S|=\lambda k:  \mbox{ any } v\in S \mbox{ has at least three neighbors in}~ \mathcal{A}\bigcap B\}~.$
\end{center}

Hence an upper bound on the probability of the event $E_k$ will be an upper
bound on the failure probability for \textbf{Majority} to reduce
the size of the set of incorrectly labeled vertices in $\mathcal{A}$
by a fraction $1- \lambda$. 

We compute now an upper bound on the probability of $E_k$. 
By the exploration process (see Definition \ref{conf_process})  it follows that 
for vertices in the set $S$, the degree sequence $\{deg_{(\mathcal{A}\bigcap B)}(v)\}_{v \in S}$  is
stochastically bounded by a vector of i.i.d.\ binomial random
variables $\{Z_v\}_{v \in S}$, i.e.,
\[
\{deg_{(\mathcal{A}\bigcap B)}(v)\}_{v \in S} \preceq \{Z_v\}_{v \in S} ~, ~~\mbox{where}~ Z_v ~ {\sim}
~ Bin(d_1,\frac{k}{n-\lambda k})~.
\]
By stochastic domination of vectors we mean the existence of a coupling of the two distributions such that the one vector is pointwise at most the other vector. 
As $\P(Z_v\geq 3)\leq \left(\frac{d_1k}{n-\lambda k}\right)^3$, by
union bound and counting the number of choices for all the possible sets
$\mathcal{A}\bigcap B$ of size $k$ and $S$ of size $\lambda k$, we obtain
the following:
\[
\P(E_k)\leq \binom{n}{k}\binom{n}{\lambda k}\Big(\frac{d_1k}{n-\lambda k}\Big)^{3\lambda k}~.
\]

Adding over all possible $k$, we obtain
\begin{eqnarray} \label{bound1}
\P\left( \left | \mathcal{A}\bigcap {B}_1 \right | \geq \lambda
k~|~k\leq \epsilon n \right)& \leq & \sum_{k=1}^{\epsilon
  n}\binom{n}{k}\binom{n}{\lambda k}\Big(\frac{d_1k}{n-\lambda
  k}\Big)^{3\lambda k} \\
 &\leq & \sum_{k=1}^{\epsilon n} \Big(\frac{d_1^{3\lambda}e^{1+\lambda}}{\lambda^{\lambda}(1-\lambda)^{3\lambda}}\Big)^k\Big(\frac{k}{n}\Big)^{(2\lambda-1)k}
\end{eqnarray}
The last inequality follows by using the bound
$\binom{n}{m}\leq\Big(\frac{ne}{m}\Big)^m$, as well as the fact that
$n - \lambda n \leq n - \lambda k$. 
Denote now by $c= c(d_1) :=\frac{d_1^{3\lambda}e^{1+\lambda}}{\lambda^{\lambda}(1-\lambda)^{3\lambda}}$.

We show now that the sum in \eqref{bound1}   is $O(n^{1/2-\lambda})$. We split this sum into two parts,
$P_1$ and $P_2$, the first representing the sum of all the terms corresponding to indices up to $\lfloor \sqrt{n} \rfloor$, and the second part
representing the rest. For $P_1$, we obtain that 
\begin{eqnarray*}
P_1 & = & \sum_{k=1}^{\lfloor \sqrt{n} \rfloor} c^k
\Big(\frac{k}{n}\Big)^{(2\lambda-1)k}\leq \sum_{k=1}^{\lfloor
  \sqrt{n} \rfloor}
c^k n^{-(\lambda-1/2)k}\\
&  \leq & \sum_{k=1}^{\infty} \left (
  \frac{c}{n^{\lambda - 1/2}} \right)^k \\
& \leq & \frac{2c}{n^{\lambda-1/2}}~.
\end{eqnarray*}
The last inequality is true for large $n$. To bound $P_2$, we note that $k/n \leq \epsilon$ and we write: 
\[
P_2 = \sum_{k=\lceil \sqrt{n} \rceil}^{\epsilon n}
c^k\Big(\frac{k}{n}\Big)^{(2\lambda-1)k}\leq \sum_{k=\lceil \sqrt{n}
\rceil}^{\infty} (c\epsilon^{2\lambda-1})^k\leq
\frac{1}{1-c\epsilon^{2\lambda-1}}(c\epsilon^{2\lambda-1})^{\lceil
  \sqrt{n} \rceil}.
\] 
The last inequality above follows by choosing $\epsilon$ so that $c \epsilon^{2 \lambda-1}
<1$. 
Hence the probability of event $E_k$ is $O
\left(n^{1/2 -\lambda}\right).$
As the problem is symmetric in $\mathcal{A}$ and $\mathcal{B}$, it
follows that a similar bound can be found for the event that
$|\mathcal{B} \bigcap A_1| > \lambda k'$. Thus by union bound,
the probability of both events is also $O\left(n^{1/2 -\lambda}\right)$,
 and the proof of the lemma is complete.
\end{proof}

\subsection{Proof of Theorem \ref{majority}}\label{proof_majority}
Let $\epsilon=\epsilon(d_1)$ as in Lemma \ref{maj_lem}. Initialize \textbf{Majority} as $({A}_0,{B}_0)=(A,B)$ where $A, B$ satisfy the conditions of Lemma \ref{maj_lem}. Denote by $({A}_i,{B}_i)$ the partition after the $i^{th}$ iteration of \textbf{Majority} where $A_i$ corresponds to the vertices labeled $+1$, i.e., $(A_i,B_i)$ is the output of the algorithm when we initialize it with $({A}_{i-1},{B}_{i-1})$.
Consider the random variables
$
X_i=\max\{ |\mathcal{A}\bigcap {B}_i|;|\mathcal{B}\bigcap {A}_i|  \}.$
Note that $\{X_i=0\}$ iff $\mathcal{A}={A}_i$ (and thus
$\mathcal{B}={B}_i$). Also by the hypothesis $X_0 \le \e n,$ so Lemma \ref{maj_lem} implies that
\[
\P(X_i\leq \lambda^i k~,\forall~ 1 \leq i )\geq 1-O(n^{1/2-\lambda}).
\]

Let now $t=\left\lceil \frac{\log(\epsilon n)^{-1}}{\log{\lambda}}
\right\rceil$. 
 Since the $X_i$s  are integer-valued random variables, we have
\[
\P(X_t=0)\geq 1-O(n^{1/2-\lambda})~,
\] 
which proves the theorem.
\qed


\noindent {\bf Acknowledgments} 
We thank Charles Bordenave for pointing out to us the connection between random lifts and the RSBM. ID acknowledges support from NSF grant DMS-08-47661. CH acknowledges support from NSF grant DMS-1308645 and NSA grant H98230-13-1-0827. GB and SG were partially supported by NSF grant DMS-08-47661 and NSF grant DMS-1308645.

\nocite{}
\bibliography{rsbm}{}
\bibliographystyle{plain}

\end{document}